\documentclass[11pt]{amsart}
\usepackage{amsmath,amsfonts, mathtools,amssymb,latexsym,amsthm,adjustbox,mathrsfs,amsbsy, bm,tikz-cd,indentfirst,makecell,graphicx, verbatim, calc, enumerate,xy}
\usepackage[top=1in,bottom=1.3in,left=1.1in,right=1.1in]{geometry}
\theoremstyle{plain}
\newtheorem{theorem}{Theorem}[section]
\newtheorem{lemma}[theorem]{Lemma}

\newtheorem{corollary}[theorem]{Corollary}
\newtheorem{proposition}[theorem]{Proposition}
\newtheorem{remark}[theorem]{Remark}
\newtheorem{definition}[theorem]{Definition}

\newtheorem{example}[theorem]{Example}
\newtheorem{point}[theorem]{}

\newcommand{\q}{\mathfrak{q}}
\newcommand{\p}{\mathfrak{p}}
\newcommand{\m}{\mathfrak{m}}

\newcommand{\R}{\mathcal{R}}
\newcommand{\F}{\mathcal{F}}
\newcommand{\Z}{\mathbb{Z}}
\newcommand{\N}{\mathbb{N}}
\newcommand{\wt}{\widetilde}

\newcommand{\Hom}{\operatorname{Hom}}

\newcommand{\Ext}{\operatorname{Ext}}

\usepackage{csquotes}
\usepackage{graphicx}
\usepackage{subfigure}
\usepackage{epsfig}
\usepackage{epstopdf}
\usepackage{float}
\theoremstyle{plain}

\usepackage{xcolor}
\usepackage{titlesec}

\titleformat{name=\section}{}{\thetitle.}{0.5em}{\centering\scshape\large}

\begin{document}
\title{\large \textbf{ Quasi-Hilbert rings and Ratliff-Rush filtrations }}
\author{\textsc{Tony J. Puthenpurakal}}
\email{tputhen@math.iitb.ac.in}

\author{\textsc{Samarendra Sahoo}}
\email{samarendra.s.math@gmail.com}

\address{Department of Mathematics, IIT Bombay, Powai, Mumbai 400 076, India}

\date{\today}


\begin{abstract}
Let $A$ be a non-Gorenstein Cohen-Macaulay ring of dimension $d\geq 1$, $I$ an ideal of $A$, and suppose $\omega_A$ is a canonical $A$-module. Set 
$$r(I,\omega_A) = \bigcup_{n \ge 0} (I^{\,n+1} \omega_A : I^{\,n} \omega_A) \subseteq A .$$
We show that the ideal $r(I,-)$ is $\omega_A$-invariant.  Motivated by this property, we introduce a new class of rings, which we call \emph{quasi-Hilbert rings}.  
We provide several examples of quasi-Hilbert rings and discuss a number of their applications. Let $A$ be a local ring with maximal ideal $\mathfrak{m}$. We prove that $A$ is quasi-Hilbert iff $\widehat{A}$ is quasi-Hilbert, where $\widehat{A}$ is the completion of $A$ w.r.t. $\mathfrak{m}.$ If $d\geq 2$ and $x\in \mathfrak{m}\setminus \mathfrak{m}^2$ is an $A\bigoplus \omega_A$-superficial element. We prove that if $A$ is quasi-Hilbert, then so is $A/(x)$. Writing $\widetilde{I}$ for the Ratliff–Rush closure of an ideal
$I$, we also provide sufficient conditions ensuring the vanishing of $r(I^n,\omega_A)/\wt{I^n}$ for all $n\geq 1.$

\end{abstract}
\maketitle

\section{Introduction}
Let $A$ be a commutative Noetherian ring with unity, $I$ an ideal of $A$, and $M$ a finitely generated $A$-module. For each $n \ge 0$, define

$$r(I^n, M) := \bigcup_{k \ge 0} (I^{n+k+1} M : I^k M) \subseteq A.$$

The ideal $r(I^n, M)$ is called the \emph{Ratliff-Rush ideal of $I^n$ with respect to $M$}; in particular, if $M = A$, we write $\widetilde{I^n} := r(I^n, A)$, which is the \emph{Ratliff-Rush closure} of $I^n$. The Ratliff–Rush ideal enjoys several important properties: The family

$$\mathcal{F}^I_M := \{ r(I^n, M) \}_{n \ge 0}$$

forms an \emph{$I$-filtration} (for the definition of $I$-filtration see \ref{Ifiltration}). If $I$ contains a non-zero divisor and ann$(M)=0$, then $\F^I_M$ is an $I$-stable filtration (see \cite{Zulfeqarr}, Theorem 5.9). Also, if $(A,\m)$ is an analytically unramified local ring with $I$ containing a non-zero divisor, then there exists a rank-one module $M$ such that

$$r(I^n, M) = \overline{I^n} \quad \text{for all } n \ge 1,$$

where $\overline{I^n}$ denotes the \emph{integral closure} of $I^n$ (see \cite{samunmixed}, Lemma 3.6).

Now, let $(A, \mathfrak{m})$ be a local ring of dimension $d$ and let $I$ be an $\mathfrak{m}$-primary ideal. Denote by
$G_I(A) := \bigoplus_{n \ge 0} I^n/I^{n+1}$
the \emph{associated graded ring} of $A$ with respect to $I$, and let $\ell(M)$ denote the length of an $A$-module $M$. The \emph{Hilbert-Samuel function} of $M$ with respect to $I$ is defined by

$$H_{M,I}(n) := \ell(M/I^n M),$$

and it is well known that there exists a polynomial $P_{M,I}(n)$ of degree $d$ such that $H_{M,I}(n) = P_{M,I}(n)$ for all sufficiently large $n$. An important property of the Ratliff--Rush ideal $K := r(I, M)$ is that it is the \emph{unique largest ideal containing $I$} for which the Hilbert-Samuel polynomial remains unchanged, that is,
$$P_{M,K}(n) = P_{M,I}(n) \quad \text{for all large } n.$$

Recall, a Noetherian ring is said to be a Hilbert ring if the Jacobson ideal is nilpotent, where the Jacobson ideal is the intersection of all maximal ideals.

Let $A$ be a Cohen-Macaulay ring and $\omega_A$ be a canonical $A$-module. In this paper, we study an invariant of canonical $A$-modules and prove that $r(I,-)$ is a $\omega_A$-invariant (see Proposition \ref{inv}). We then prove:
\begin{theorem}[Theorem \ref{Hilb}]\label{1}
     Let $A$ be a Cohen-Macaulay ring, $J$ be the Jacobson ideal of $A,$ and $\omega_A$ be a canonical $A$-module. Then TFAE
     \begin{enumerate}
         \item $A$ is a Hilbert ring.
         \item $r(J,\omega_A)=A$.
         \item $r(J^n,\omega_A)=A$ for all $n\geq 1.$
     \end{enumerate}
\end{theorem}
In particular, the Rees algebra $\R(\F^J_{\omega_A})=\bigoplus_{n\geq 0}r(J^n,\omega_A)t^n=A[t]$ is a Noetherian ring but not a finitely generated $\R^J$-module. Let $\wt{\R}^J=\bigoplus_{n\geq 0}\wt{J^n}$ be the Rees algebra with respect to the Ratliff-Rush filtration $\{\wt{J^n}\}_{n\geq 0}.$ If $J$ is a regular ideal (contains a non-zero divisor), then $\wt{\R}^J$ and $\R(\F^J_{\omega_A})$ are finitely generated $\R^J$-module (see Remark \ref{rmk:tildeI^nM}(1), \ref{remark}(3)). Thus, $\R(\F^J_{\omega_A})/\wt{\R}^J$ is a  finitely generated $\R^J$-module and of dimension $\leq \dim(A)$ (see Lemma \ref{up}). Also note that if $A$ is Gorenstein then $\R(\F^J_{\omega_A})/\wt{\R}^J=0$ (see Remark \ref{remark}(d)).

Motivated by this, we define \emph{quasi-Hilbert} ring as follows:  Let $A$ be a Cohen-Macaulay ring of dimension $d\geq 1$ and $\omega_A$ be a canonical $A$-module. Let  $J\neq 0$ be the Jacobson radical ideal of $A$ that contains a non-zero divisor. We say $A$ is a \emph{quasi-Hilbert} ring if the Krull dimension of $\R(\F^J_{\omega_A})/\wt{\R}^J$ is equal to $\operatorname{dim}(A)$. In particular, a quasi-Hilbert ring is not Gorenstein. 

Note that the larger the Krull dimension of the graded ring $\R{(\F^J_{\omega_A})}/\wt{\R}^J$, the closer the ring $A$ behaves to a Hilbert ring.

We prove:
\begin{theorem}[Theorem \ref{Thm:main}]\label{first}
 Let $A$ be a semi-local Cohen–Macaulay ring of dimension $d\geq 1$, which is neither Gorenstein nor a Hilbert ring. Suppose $\omega_A$ is a canonical $A$-module. Let $J$ be the Jacobson ideal of $A$. Set $E=\R(\F^J_{\omega_A})/\wt{\R}^J$. Then the following holds.
  \begin{itemize}
  \item[{(a)}] If $J$ is regular then the function $n\mapsto \ell{(E_n)}$ is of polynomial type.
\item[{(b)}] If $A$ is a local ring with $G_J(A)$ unmixed then either $E=0$ or the function $n\mapsto \ell(E_n)$ coincides with a polynomial of degree $d-1$. In particular, if $E\neq 0$ then $A$ is a quasi-Hilbert ring.
  \end{itemize}
  \end{theorem}
We also provide an example that demonstrates the necessity of the unmixedness condition of $G_I(A)$ in Theorem \ref{first} (see Example \ref{ex}). As an application of a quasi-Hilbert ring, we prove:

\begin{theorem}[Corollary \ref{max}] \label{2}
     Let $(A,\m)$ be a Cohen-Macaulay local domain of dimension $d\geq 1$. Suppose $\omega_A$ is the canonical module of $A.$ Let $M$ be a maximal Cohen-Macaulay $A$-module and $N_i=\operatorname{Syz}^A_i(M)$ be the $i$-th syzygy of $M$.  If $A$ is a quasi-Hilbert ring then the function  $n\mapsto \ell(r(\m^n,N^{\dagger}_i)/\wt{\m^n})$ coincides with a polynomial of degree $d-1$ for all $i\geq 1$, where $N^{\dagger}_i=\Hom(N_i,\omega_A).$
\end{theorem}
We also show that the quasi-Hilbert property of $A$ is preserved under passage to the completion of $A$ and upon reduction modulo a superficial element (for the definition of superficial element see \ref{sup}).

\begin{theorem}[Theorem \ref{hat}, \ref{modx}]\label{3}
     Let $(A,\m)$ be a Cohen-Macaulay local ring of dimension $d\geq 1$. Suppose $\omega_A$ is a canonical $A$-module.
     \begin{enumerate}
         \item $A$ is quasi-Hilbert iff  $\hat{A}$ is quasi-Hilbert, where $\hat{A}$ is the completion of $A$ w.r.t. $\m$.
         \item Let $d\geq 2$. Assume $A$ has an infinite residue field. Let $x\in \m\setminus\m^2$ be an $A\bigoplus \omega_A$-superficial element. If $A$ is quasi-Hilbert, then so is $A/(x)$.
     \end{enumerate}
         
\end{theorem}
For a regular ideal $I$ and a maximal Cohen-Macaulay $A$-module $M$, we investigate the relationship between the $r(I^n,M)$ and the Ratliff-Rush closures $\widetilde{I^n}$. In particular, we provide sufficient conditions guaranteeing that $r(I^n,M)=\widetilde{I^n}$ for all $n \ge 1$ as follows:

\begin{theorem}[Theorem \ref{v2}] \label{last}
    (with the same hypothesis of Theorem \ref{vanish}) If $G_I(A)$ is unmixed then the following holds.
    \begin{enumerate}
        \item $E=\bigoplus_{n\geq 1}r(I^n,M)/\wt{I^n}$ has dimension equal to $d$ or $E=0$ as $\R^I$-module.
        \item Suppose $E\neq 0.$ Let $\q$ be a prime ideal of $A$ containing $I.$ Then the $\dim(E_{\q})=\operatorname{ht}(\q)$ as a module over $\widehat{{\R}^{I_{\q}}}=\bigoplus_{n\in \Z}{I^n_{\q}}t^n,$ where $\wt{I^n_{\q}}=A_{\q}$ for all $n\leq 0.$ 
        \item If $M_{\q}$ is free for some prime $\q$ containing $I$ then $r(I^n,M)=\wt{I^n}$ for all $n\geq 1.$ 
    \end{enumerate}
\end{theorem}

We now outline the structure and main results of the paper. In Section~2, we present the necessary preliminaries. In Section~3, we first prove that $r(I,-)$ is an invariant on canonical modules over Cohen-Macaulay rings, and then we sketch the proofs of Theorem~\ref{1} and Theorem~\ref{first}. In Section~4, we provide an example illustrating the necessity of the unmixedness condition on $G_I(A)$ in Theorem~\ref{first}. In Section~5, we prove Theorem~\ref{2} and Theorem~\ref{3}. Finally, in Section~6, we prove Theorem \ref{last}.

\section{Preliminaries}
Throughout this paper, all rings considered are Noetherian, and all modules considered, unless stated otherwise, are finitely generated. We will use \cite{Bruns} as a general reference.
Also, $\ell(M)$ denotes the length of an $A$-module.
\begin{point}\label{sup}
 \normalfont
 Let $(A, \m)$ be a local ring.
   An element $x\in  I$ is called $A$-superficial with respect to $I$ if there exists $c\in \N$ such that for all $n \geq c$, $$(I^{n+1}:x)\cap I^c = I^n.$$ If depth$(I, A)>0$, then one can show that an $A$-superficial element is $A$-regular. Furthermore, in this case $$(I^{n+1}:x) = I^n \text{ for all } n\gg0.$$ Superficial elements exist if the residue field is infinite, see \cite[Page 86]{hccmm}. 
\end{point}

\begin{point}\label{Ifiltration}
\normalfont
 A filtration of ideals $\F=\{I_n\}_{n\geq 0}$ is said to be $I$-filtration if it satisfies the following conditions:
 \begin{enumerate}
     \item $I_0=A$ and $I_{n+1}\subseteq I_n$ for all $n\geq 0.$
     \item $I\subseteq I_1$ and $I_1\neq A.$
     \item $I_nI_m\subseteq I_{n+m}$ for all $n,m\geq 0$.
 \end{enumerate}

\end{point}
 In addition, if $II_n=I_{n+1}$ for all $n\gg0$ then the $I$-filtration $\F$ is called as an $I$-stable filtration. It is well known that if $\F$ is an $I$-stable filtration, then  the Rees algebra $\R(\F)=\bigoplus_{n\geq 0}I_nt^n$ is a finite module over $\R=A[It]=\bigoplus_{n\geq 0}I^nt^n$.
\begin{point}
\normalfont
 Let $(A,\m)$ be a Noetherian local ring. Let $\F = \{I_n\}_{n\geq 0}$ be an $I$-stable filtration of $A$. Set ${L}^{\F}=\bigoplus_{n\geq 0}A/{I_{n+1}}$. Let ${\R}(\F)=\bigoplus_{n\geq 0}{I_n}t^n$ be the Rees ring with respect to the filtration $\{{I_n}\}_{n\geq 0}$. As ${\R}(\F)$ is a subring of $A[t]$, so $A[t]$ is a ${\R}(\F)$-module. Note that we have an exact sequence of $\R(\F)$-modules $$0\to {\R}(\F)\to A[t]\to {L}^{\F}(-1)\to 0 .$$ It follows that  ${L}^{\F}(-1)$ is a $\R(\F)$-module. Hence   ${L}^{\F}$ is a $\R$-module, for more details see (\cite{Part1}, 4.2).
 Note $L^\F$ is \emph{not} a finitely generated $\R(\F)$-module.
\end{point}

\begin{point}\label{rit}
\normalfont
Let $A$ be a ring and $I$ an ideal of $A$. Let $M$ be an $R$-module. Consider the following ascending chain of ideals in $A$ $$I\subseteq (IM:M)\subseteq (I^2M:IM)\subseteq (I^3M:I^2M)\subseteq \cdots \subseteq(I^{n+1}M:I^nM) \cdots .$$ Since $R$ is Noetherian, this chain stabilizes. The stable value is denoted as $r(I, M)$. It is easy to prove that the filtration of ideal $\mathcal{F^I_M}=\{r(I^n, M)\}_{n\geq 0}$ forms an $I$-filtration, see (\cite{Zulfeqarr}, Theorem 2.1). Define $$s(r(I^t,M))=\text{min}\{ n\, |\, r{(I^t,M)}=(I^{n+t}M:_AI^nM)\} \text{ for all $t\geq 1$}.$$
\end{point}

\begin{point}\label{rho}
\normalfont
Consider the following chain of submodules of M:$$IM\subseteq (I^2M:_MI)\subseteq (I^3M:_MI^2)\subseteq \cdots \subseteq (I^{n+1}M:_MI^n) \cdots .$$ Since M is Noetherian, this chain of submodules stabilizes. The stable value is denoted by $\widetilde{IM}$. Note that the filtration $\{\widetilde{I^nM}\}_{n\geq 1}$ is an $I$-filtration. Set $$\rho(I^t,M)=\text{min}\{ n\, |\, \wt{I^tM}=(I^{n+t}M:_MI^n)\} \text{ for all $t\geq 1$}.$$
\end{point}

\begin{remark}
\label{rmk:tildeI^nM}
\normalfont
\begin{enumerate}
    \item If grade$(I, M)>0$, then for all $n\gg0$,  $I^nM=\widetilde{I^nM}.$ In particular, $\{\widetilde{I^nM}\}_{n\geq 1}$ is an $I$-stable filtration (see \cite[2.2]{tildeI^nM}). Set $$s^*(I,M)=\text{min}\{ n\, |\, \wt{I^mM}=I^mM \text{ for all }m\geq n\}.$$
    \item If grade$(G_I(A)_+, G_I(M))>0$ then $I^nM=\widetilde{I^nM}$ for all $n\geq 1$, see \cite[fact 9]{Heinzer}.
\end{enumerate}
\end{remark}

\begin{lemma}[\cite{Zulfeqarr}, Theorem 5.9]
  \normalfont
  \label{Thm:Istable}
      Let $I$ be a regular ideal of $A.$ If ann$(M)=0$ then $\mathcal{F^I_M}=\{r(I^n, M)\}_{n\geq 0}$ is an $I$-stable filtration.
  \end{lemma}

 Let $\R^I=A[It]$ be the Rees ring. By Remark \ref{rmk:tildeI^nM} and Lemma \ref{Thm:Istable}, if $I$ is regular and ann$(M)=0$ then $\R({\F^I_M})=\bigoplus_{n\geq 0}r(I^n,M)t^n$ and $\wt{\R}^I =\bigoplus_{n\geq 0}\wt{I^n}t^n$ are finitely generated $\R$-module. Hence the quotient module $\dfrac{\R({\F^I_M})}{\wt{\R}^I}=\bigoplus_{n\geq 0}\dfrac{r(I^n,M)}{\wt{I^n}}$ is a finitely generated $\R^I$-module.

 \begin{definition}\cite[Theorem 1]{kishor}
\normalfont
    Let $(A,\m)$ be a quasi-unmixed local ring of dimension $d\geq 1$ with infinite residue field. Let $I$ an $\m$-primary ideal of $A.$ Then there exist unique largest ideals $I_{(k)}$ for $1\leq k\leq d$ containing $I$ such that 
    \begin{enumerate}
        \item $e_i(I)=e_i(I_{(k)})$ for $0\leq i\leq k$, where $e_i(I)$ is the $i$-th Hilbert coefficient of $G_I(A)$ (see \cite{hccmm}, for the definition of Hilbert coefficient).
        \item $I\subseteq I_{(d)}\subseteq \ldots \subseteq I_{(1)}\subseteq \overline{I}=\text{ integral closure of }I.$
    \end{enumerate}
    The ideal $I_{(k)}$ is called as $k$-th coefficient ideal of $A.$
\end{definition}

 We will use the following result to prove our main theorem. For the proof of the result, see (\cite{samunmixed}, Proposition 2.9).  
\begin{proposition}
\label{UE}
     Let $(A,\m)$ be a Noetherian ring of dimension $d$ and $I$ an ideal of $A.$
    \begin{enumerate}[\rm (1)]
     \item  If $A$ is equidimensional and universal catenary then $G_{I}(A)$ and $\wt{G}_I(A)$ are equidimensional.
        \item Let $(A,\m)$ be a local ring. If $A$ is quasi unmixed with infinite residue field  and $I$ is $\m$-primary then  $G_{I}(A)$ is unmixed iff $I^n=I^n_{(1)}$ for every $n$, where $I^n_{(1)}$ is the first coefficient ideal of $I^n$.
    \end{enumerate}
\end{proposition}

Note that for the first part, we don't need $A$ to be local (see \cite{Bruns}, Lemma 4.5.5). The same proof will work for equidimensional property of $\wt{G}_I(A).$

\begin{point}
    \normalfont
   Let $(A,\m)$ be a local ring and $I$ an ideal of $A.$ A reduction of $I$ is an ideal $K \subseteq I$ such that $KI^n =I^{n+1}$ for some $n \in N$. If $K$ is a reduction of $I$, the reduction number of $I$ with respect to $K$ is defined as: $$\text{r}_K(I) =\text{min}\{n\, |\, KI^n = I^{n+1}\}.$$ A reduction is minimal if it is minimal with respect to inclusion. The reduction number of $I$ is defined as $\text{r}(I) = \text{min}\{\text{r}_K(I) \, |\, K \text{ is a minimal reduction of }  I\}$. 
\end{point}

\section{Invariant of canonical modules over Cohen-Macaulay rings}
In this section, we introduce the notion of a quasi-Hilbert ring and provide a characterization of such rings over a semi-local Cohen–Macaulay domain.

Let $A$ be a Cohen-Macaulay ring that admits a canonical $A$-module. It is well known that a canonical module of a Cohen-Macaulay ring is unique up to a tensor product with a locally free module of rank one (see \cite{Bruns}, Remark 3.3.17).

\begin{proposition}\label{inv}
    Let $A$ be a Cohen-Macaulay ring and $I$ an ideal of $A$. Then the ideal $r(I,-)$ is an invariant on canonical $A$-modules.
\end{proposition}
\begin{proof}
    Let $w^1_A$ and $w^2_A$ be any two canonical $A$-modules. Since the canonical module of $A$ is unique up to a tensor product with a locally free module of rank one, we have $w^1_A\cong w^2_A\otimes M$, where $M$ is a locally free $A$-module of rank one. By (\cite{Zulfeqarr}, Proposition 1.4(c)), it follows that $$r(I,w^1_A)\cong r(I,w^2_A\otimes M)\supseteq r(I,w^2_A)+r(I,M).$$ By Corollary 1.6 of \cite{Zulfeqarr}, we have $r(I,w^1_A)\supseteq r(I,w^2_A)+\wt{I}=r(I,w^2_A).$ Similarly, we can prove the other containment by using $w^2_A\cong w^1_A\otimes N$ for some locally free $A$-module of rank one.
\end{proof}
Recall that a ring $A$ is said to be Hilbert if the Jacobson radical ideal is nilpotent. From now on, we denote $J$ as the Jacobson radical ideal of $A.$
\begin{theorem}\label{Hilb}
     Let $A$ be a Cohen-Macaulay ring and $\omega_A$ be a canonical $A$-module. Then TFAE
     \begin{enumerate}
         \item $A$ is a Hilbert ring.
         \item $r(J,\omega_A)=A$.
         \item $r(J^n,\omega_A)=A$ for all $n\geq 1.$
     \end{enumerate}
\end{theorem}
\begin{proof}
    $(1)\implies(2),(3).$ Suppose $A$ is a Hilbert ring. Then $J^n=0$ for some $n\geq 1.$ The result follows  since $r(J^i,\omega_A)=(J^{i+n}\omega_A:J^{n}\omega_A)$ for all $n\gg0$ and $i\geq 1.$ 

    $(2),(3)\implies(1).$ It suffices to prove $(2)\implies (1)$. Suppose $r(J,\omega_A)=A$. Let $r(J,\omega_A)=(J^{k+1}\omega_A:J^{k}\omega_A)$ for some $k\geq 1.$ By the assumption, we have $J^{k+1}\omega_A=J^{k}\omega_A$. This implies that there exists $\alpha\in J$ such that $(1+\alpha)J^k\omega_A=0$. Thus, $J^{k}\omega_A=0.$ As $\omega_A$ is a faithful $A$-module (see \cite{Bruns}, Theorem 3.3.13), we have $J^{k}=0.$ This completes the proof. 
\end{proof}

 In general, we know that $I\subseteq r(I,\omega_A)$ for any ideal $I$. But the next result says that the Jacobson radical ideal of a Cohen-Macaulay semi-local ring is Ratliff-Rush closed with respect to the canonical module.

 \begin{proposition}
     Let $A$ be a semi-local Cohen-Macaulay ring and $\omega_A$ be a canonical $A$-module. Then $r(J,\omega_A)=J$.
 \end{proposition}
 \begin{proof}
     Let $\m_1,\ldots,\m_n$ be the maximal ideals of $A$. Therefore $J=\m_1\cap\ldots \cap \m_n.$ For $1\leq i\leq n$,

     \begin{equation*}
         \begin{split}
           J_{\m_i}=\m_iA_{\m_i} \subseteq & r(J,\omega_A)_{\m_i}\\ 
           = & r(J_{\m_i},w_{A_{\m_i}})\\ 
           = & r(\m_iA_{\m_i},w_{A_{\m_i}})\\
           \subseteq & \overline{\m_iA_{\m_i}} \\ = &\m_iA_{\m_i},
         \end{split}
     \end{equation*}
where $\overline{\m_iA_{\m_i}} $ is the integral closure of ${\m_iA_{\m_i}}$. The last equality is due to the fact that the maximal ideal in a local ring is integrally closed. This implies that $r(J,\omega_A)_{\m_i}=J_{\m_i}$ for all $1\leq i\leq n.$ This proves the result.
 \end{proof}

\begin{remark}\label{remark}
 \normalfont 
  Let $A$ be a Cohen-Macaulay ring of dimension $d\geq 1$. Suppose $\omega_A$ be a canonical $A$-module.
 \begin{enumerate}
     \item Let $I$ a regular ideal of $A$. Since $\omega_A$ is a faithful $A$-module (see \cite{Bruns}, Thoerem 3.3.13) it follows that ann$(\omega_A)=0.$ By Theorem 5.3 of \cite{Zulfeqarr}, the filtration $\F^I_{\omega_A}$ is $I$-stable. Hence, the Rees module $\R(\F^I_{\omega_A})$ is a finitely generated $\R^I$-module. 
     \item Set $S=\R(\F^J_{\omega_A})=\bigoplus_{n\geq 0}r(J^n,\omega_A)t^n$. Note that $S$ is a $\R^J$-module. If $A$ is a Hilbert ring, then $S=A[t]$ is a Noetherian ring but not a finitely generated $\R^J$-module.
     \item If $J$ is a regular ideal, then by Theorem 5.3 of \cite{Zulfeqarr}, $S$ is a finitely generated $\R^J$-module. For instance, if $A$ is a domain and $J\neq 0$ then $J$ is regular.
     \item If $A$ is a Gorenstein ring then $w_{A_{\p}}\cong A_{\p}$ for all $\p\in Spec(A).$ This implies that $r(I^n,\omega_A)=\wt{I^n}$ for all $n\geq 1$ and any ideal $I.$ 
 \end{enumerate}

\end{remark}
\begin{lemma}\label{up}
    Let $A$ be a Cohen-Macaulay ring of dimension $d\geq 1$ and $\omega_A$ be a canonical $A$-module. Let $J\neq 0$ be the Jacobson radical ideal of $A$ containing a non-zero divisor. Then the Krull dimension of $S/\wt\R^J\leq d.$ 
\end{lemma}
\begin{proof}
    From the exact sequence $$0\to S/\wt\R^J=\bigoplus_{n\geq 0}r(J^n,\omega_A)/\wt{J^n}\to \bigoplus_{n\geq 0}A/J^n,$$ we obtain Ass$_{\R^J}(S/\wt\R^J)\subseteq \text{Ass}_{\R^J}(\wt{L}^J(A)).$ By Lemma 3.3 of \cite{samunmixed}, $\text{Ass}_{\R^J}(\wt{L}^J(A))=\text{Ass}_{\R^J}(\wt{G}_J(A))$. By (\cite{Bruns}, Theorem 4.5.6), it follows that $$\dim_{\R^J}(S/\wt\R^J)\leq \dim_{\R^J}(\wt{G}_J(A))=d.$$
\end{proof}
We now define a quasi-Hilbert ring.

\begin{definition}
   Let $A$ be a Cohen-Macaulay ring of dimension $d\geq 1$ and $\omega_A$ be a canonical $A$-module. Let  $J\neq 0$ be the Jacobson radical ideal of $A$ containing a non-zero divisor. We say $A$ is a quasi-Hilbert ring if the Krull dimension of $S/\wt{\R}^J$ is equal to $d$.
\end{definition}

 The next theorem provides sufficient conditions for $A$ to be a quasi-Hilbert ring.

\begin{theorem}
\label{Thm:main}
 Let $A$ be a non-Gorenstein semi-local Cohen–Macaulay ring of dimension $d\geq 1$ and $J$ be the Jacaobson ideal of $A$.  Suppose $\omega_A$ be a canonical $A$-module.  Set $E=S/\wt{\R}^J$. Then the following holds.
  \begin{itemize}
  \item[{(a)}] If $J$ is a regular ideal then the function $n\mapsto \ell{(E_n)}$ is of polynomial type.
      \item[{(b)}] If $A$ is a local ring with $G_J(A)$ unmixed then either $E=0$ or the function $n\mapsto \ell(E_n)$ coincides with a polynomial of degree $d-1$. In particular, if $E\neq 0$ then $A$ is a quasi-Hilbert ring.
  \end{itemize}
  \end{theorem}

  \begin{proof}
       (a) By Remark \ref{remark}(3), $E$ is a finitely generated $\R^J$-module. Note that $E_n=r(J^n,\omega_A)/\wt{J^n}$ for each $n\geq 0$. Let $\p$ be a prime ideal of $A$ not in $\m Spec(A)$. Now, $$(E_n)_{\p}=(r(J^n,\omega_A)/\wt{J^n})_{\p}=r(J^n_{\p},w_{A_\p})/\wt{J^n_{\p}}=0.$$ Therefore, $\ell(E_n)<\infty$ for all $n\geq 0.$ It follows that the function $n\mapsto \ell{(E_n)}$ is of polynomial type.
       
       (b) Since $A$ is local, $J=\m$ contains a non-zero divisor. If $E=0$, then we have nothing to show. Suppose $E\neq 0$. By the Proposition \ref{UE}, we have $G_J(A)$ is unmixed and equidimensional. Since the dimension of $A$ is positive, we have depth $G_J(A)>0$. By Remark \ref{rmk:tildeI^nM}(2), $\wt{J^n}=J^n$ for all $n\geq 0$. It follows that $\wt{\R^J}=\R^J$. The result follows due to Theorem 3.4 of \cite{samunmixed}. 
  \end{proof}

The next result says that Theorem \ref{Thm:main} is true for an arbitrary ideal under some extra conditions on the ring.

\begin{theorem}
\label{Thm:main1}
     Let $A$ be a non-Gorenstein semilocal Cohen-Macaulay ring of dimension $d\geq 1$ with $A_{\p}$ is Gorenstein for all $\p\notin \m {Spec}(A)$. Let $I$ be a regular ideal of $A$ and $\omega_A$ a canonical $A$-module.
      Set $E=\R(\F^I_M)/\wt{\R}^I$. Then the following holds.
  \begin{itemize}
  \item[{(a)}] The function $n\mapsto \ell{(E_n)}$ is of polynomial type.
      \item[{(b)}] If $A$ is a local ring with $G_I(A)$ unmixed then either $E=0$ or the function $n\mapsto \ell(E_n)$ coincides with a polynomial of degree $d-1$.
  \end{itemize}
\end{theorem}
\begin{proof}
       (a) By Remark \ref{remark}(1), $E$ is a finitely generated $\R^I$-module. Note that $E_n=r(I^n,w)/\wt{I^n}$ for each $n\geq 0$. Let $\p$ be a prime ideal of $A$ not in $\m$-Spec$A$. Since $A_{\p}$ is Gorenstein we have $$(E_n)_{\p}=(r(I^n,\omega_A)/\wt{I^n})_{\p}=r(I^n_{\p},w_{A_\p})/\wt{I^n_{\p}}=r(I^n_{\p},A_{\p})/\wt{I^n_{\p}}=0.$$ Therefore, $\ell(E_n)<\infty$ for all $n\geq 0.$ It follows that the function $n\mapsto \ell{(E_n)}$ is of polynomial type.
       
       (b) If $E=0$, then we have nothing to show. Suppose $E\neq 0$. By the Proposition \ref{UE}, we have $G_I(A)$ is unmixed and equidimensional. It follows that $\wt{\R^I}=\R^I$. The result follows due to Theorem 3.4 of \cite{samunmixed}. 
  \end{proof} 

We now give some examples that satisfy the hypothesis of Theorem \ref{Thm:main} and \ref{Thm:main1}. We have used both Singular and Macaulay 2 to verify these examples. 

\begin{remark}[\cite{Bruns}, Theorem 3.3.7(b)]
    Let $R$ be a Cohen-Macaulay local ring that admits a canonical module (say $w_R$) and $I$ an ideal of $R$. Let $A=R/I$ be a Cohen-Macaulay ring. Then $\omega_A\cong \Ext^t_R(A,R)$, where $t=\dim(R)-\dim(A)$.
\end{remark}

\begin{example}
\normalfont
Let $A=k[[t^{11},t^{17},t^{18}]]$ be a subring of $k[t]$, where $k$ is a field. Using singular, it can be checked that $A\cong k[[x,y,z]]/I$, where $I=(x^3z-y^3,x^5y-z^4,x^8-y^2z^3).$ Note that $(A,\m)$ is a Cohen-Macaulay local domain and $A/x\cong k[[y,z]]/(y^3,z^4,y^2z^3)$. The Hilbert series of $A/x$ is $$H_{A/x}(t)=1+2t+3t^2+3z^3+2z^4,$$ which is not symmetric. By Corollary 4.4.6(a) of \cite{Bruns}, $A$ is not a Gorenstein ring. Note that the canonical module ($\omega_A$) of $A$ is isomorphic to a submodule of $A^2$ and is generated by $\big((y,z),(x^3,y^2),(z^3,x^5)\big).$ Using singular, we can also check that $$(\m^{65}\omega_A:\m^{59}\omega_A)=(z^4,y^2z^3,xyz^3,y^3z^2,y^4z,xy^3z,y^5,xy^4,x^2y^3,x^2y^2z^2,x^4y^2,x^6)$$ and $$\m^6=(z^4,y^3z^2,y^4z,xy^3z,y^5,xy^4,x^2y^3,xy^2z^3,x^2yz^3,x^2y^2z^2,x^4y^2,x^6).$$ Using Macaulay 2, it can be checked that  $G_{\m}(A)$ is unmixed and $xyz^3\notin \m^6$. Thus $r{(\m^6,\omega_A)}\neq \m^6$.  This satisfies the hypothesis of Theorem \ref{Thm:main}. 
\end{example}

\begin{example}
\normalfont
    Let $A= k[[x,y,z]]/I$, where $I=(x^3z-y^3,x^5y-z^4,x^8-y^2z^3).$  By the previous example, $A$ is a non-Gorenstein local Cohen--Macaulay domain. The only prime ideals of $A$ are $(0)$ and $\mathfrak m$. Since $A$ is a domain, the localization $A_{(0)}$ is a field and hence Gorenstein.

Let $K=(x,y)$. Clearly, $K \neq \mathfrak m$. Using \textsc{Singular}, one checks that
\[
(K^2 \omega_A : K\omega_A) = (x,y,z^3) \neq K,
\]
and therefore $r(K,\omega_A) \neq K$. Moreover, a computation in \textsc{Macaulay2} shows that the associated graded ring $G_K(A)$ is unmixed. Thus, the hypotheses of Theorem~\ref{Thm:main1} are satisfied.

\end{example}
In Singular, to compute isomorphisms of $A$, canonical modules, and colon ideals, we use the following commands.

\begin{verbatim}
    ring R=0,(x,y,z),ds;
    ideal m=x,y,z;
    ring S=0,t,ds;
    map g=R,t^11,t^17,t^18;
    setring R;
    ideal I=kernel(S,g);
    I; 
    matrix m[1][5]=x^3z-y^3,x^2y^4-z^5,x^5y-z^4,y^7-xz^6,x^8-y^2z^3;
    matrix N=m;
    Ext_R(2,N);
    module M1=Ext_R(2,N);
    qring A=std(I);
    setring A;
    module w=imap(R,M1);
    ideal m=x,y,z;
    ideal D=quotient(m^65*w,m^59*w);
    D;
\end{verbatim}
\enquote{$I;$} will give the generators of $I=(x^3z-y^3,x^2y^4-z^5,x^5y-z^4,y^7-xz^6,x^8-y^2z^3)$. \enquote{matrix N=m} will give the cokernel of the map $\R^5\xrightarrow{[m]}R$, i.e. $R/I.$ \enquote{$\Ext_R(2,N);$} will calculate $\Ext^2_R(A/I,R)\cong \omega_A$. \enquote{qring A=std(I);} and \enquote{setring A;} will define the quotient ring $A=R/I$. \enquote{module w=imap(R,M1);} will consider $w=\Ext_R(2,N)$ as $A$-module. Finally, \enquote{ideal D=quotient($m^{65}*w,m^{59}*w$);} will calculate the colon ideal $(\m^{65}w:\m^{59}w).$

To check the unmixedness of $G_{\m}(A)$ and $G_K(A)$ in the above examples, we have used the following commands in Macaulay 2.
\begin{verbatim}
    R=QQ[x,y,z]
    I=ideal(x^3z-y^3,x^5y-z^4,x^8-y^2z^3)
    A=R/I
    m=ideal(x,y,z)
    reesAlgebra(m)
    T=reesAlgebra(m)
    mT=sub(m,T)
    W=T/mT
    prune W
    ideal oo
    associatedPrimes oo
\end{verbatim}
\enquote{associatedPrimes oo} will give the associated primes of $W=G_{\m}(A)$. Replacing $\m$ by $K$, we get the associated primes of $W=G_K(A)$. To compute a minimal generating set of an ideal and to check whether an element 
$f$ of the ring belongs to the ideal, we use the following commands.

\begin{verbatim}
    R=QQ[x,y,z]
    I=ideal(x^3z-y^3,x^5y-z^4,x^8-y^2z^3)
    A=R/I
    m=ideal(x,y,z)
    mingens m^6
    f=x*y*z^3
    member(f,m^6)
\end{verbatim}

\section{Some more analysis on $r(I^n,M)$ and $\wt{I^nM}$}
In this section, we discuss the relation between $\rho(I^n,M)$, $s^*(I^n,M)$ and $s(r(I^n,M)).$ At the end of this section, using these relations, we give an example which says that the unmixedness condition on $G_{\m}(A)$ can't be removed from the hypothesis of Theorem \ref{Thm:main}.

Recall (see \ref{rit}, \ref{rho}, \ref{rmk:tildeI^nM}),
 $$s(r(I^t,M))=\text{min}\{ n\, |\, r{(I^t,M)}=(I^{n+t}M:_AI^nM)\} \text{ for all $t\geq 1$},$$

$$\rho(I^t,M)=\text{min}\{ n\, |\, \wt{I^tM}=(I^{n+t}M:_MI^n)\} \text{ for all $t\geq 1$},$$

and $$s^*(I,M)=\text{min}\{ n\, |\, \wt{I^mM}=I^mM \text{ for all }m\geq n\}.$$

The next result compares the inequality between $\rho(I^n,M)$ and $s^*(I^n,M).$

\begin{proposition}\label{ref}
    Let $A$ be a ring, $I$ a regular ideal of $A$, and $M$ be an $A$-module. Then 
    \begin{itemize}
        \item[{(a)}] $\rho(I^n,M)\leq \rho(I^{n+1},M)+1$ for all $n\geq 1.$
        \item[{(b)}]  $\rho(I^n,M)\leq s^*(I,M)-n$ for all $n<s^*(I,M).$
    \end{itemize}    
\end{proposition}
\begin{proof}
   (a) Let $\rho(I^{n+1},M)=t$. This implies that $\wt{I^{n+1}M}=I^{n+1+t}M:I^{t}.$ Now, we have
     \begin{equation*}
        \begin{split}
        (I^{n+1+t}M:I^{t+1})= &((I^{n+1+t}M:_MI^{t}):_AI) \\ = &((I^{n+2+t}M:_MI^{t+1}):_AI)\\ = & (I^{n+2+t}M:I^{t+2}).
        \end{split}
    \end{equation*}
    Continuing this process, we get $\wt{I^{n}M}=(I^{n+1+t}M:_MI^{t+1}).$ This proves part(a).
   
   (b)  Let $s^*(I,M)=t.$ This implies that $\wt{I^nM}=I^nM$ for all $n\geq t.$ Note that $\wt{I^nM}=I^nM$ implies that $\rho(I^n,M)=0$. By part(a),
    \begin{equation*}
        \begin{split}
            \rho(I^n,M) & \leq \rho(I^{n+1},M)+1\\ 
            & \leq \ldots \\ 
            & \leq \rho(I^{n+t-n-1},M)+t-n-1 \\ 
            &\leq \rho(I^t,M)+t-n \\ 
            & =t-n.   
        \end{split}
    \end{equation*}
     This proves part(b).
\end{proof}

\begin{proposition}\label{bound}
     Let $A$ be a ring, $I$ a regular ideal of $A$, and $M$ be an $A$-module such that grade$(I,M)>0$. Then $s(r(I^n,M))\leq \rho(I^{n},M)\leq s^*(I,M)-n$ for all $n\geq 1.$
\end{proposition}
\begin{proof}
     Let $\rho(I^{n},M)=t$. This implies that $\wt{I^{n}M}=I^{n+t}M:I^{t}.$ Now, we have
          \begin{equation*}
        \begin{split}
       (I^{n+t}M:I^{t}M)= & ((I^{n+t}M:_MI^{t}):_AM) \\ = &((I^{n+1+t}M:_MI^{t+1}):_AM)\\ = & (I^{n+1+t}M:_AI^{t+1}M).
        \end{split}
    \end{equation*}
     
      Continuing the same process, we get $r{(I^{n},M)}=I^{n+t}M:I^{t}M.$ This proves the first inequality. The second inequality follows from the Proposition \ref{ref}(b).
\end{proof}
The next result is the module analogue of the theorem due to Rossi and Swanson. Although the proof is essentially the same, we provide it here for clarity.
\begin{proposition}\label{rossi}
    Let $A$ be a local ring and $I$ an ideal of $A$ having principal reduction such that grade$(I,M)>0$. Then $s^*(I,M)\leq r(I)$.
\end{proposition}
\begin{proof}
    Let $(x)$ be the minimal reduction of $I$ and $r(I)=r$. We have 
    \begin{equation*}
        \begin{split}
            I^rM\subseteq\wt{I^rM}= & (I^{r+n}M:I^n) \\ \subseteq & (I^{r+n}M:x^n) \\ = & (x^nI^{r}M:x^n) \\ =&I^rM
        \end{split}
    \end{equation*}
This proves the result.
\end{proof}
We now discuss the example, which states that the hypothesis $G_{\m}(A)$ being unmixed cannot be dropped in Theorem \ref{Thm:main}. 

\begin{example}\label{ex}
\normalfont
    Let $A=k[[t^{5},t^{6},t^{13}]]$ be a subring of $k[t]$, where $k$ is a field. Using singular, it can be checked that $A\cong k[[x,y,z]]/I$, where $I=(xz-y^3,z^2-x^4y,y^2z-x^5,y^5-x^6).$ Note that $(A,\m)$ is a local Cohen-Macaulay domain and $A/x\cong k[[y,z]]/(y^3,z^2,y^2z)$. The Hilbert series of $A/x$ is $$H_{A/x}(t)=1+2t+2t^2.$$ Since the socle dimension of $A/x$ is greater than or equal to $2$, $A/x$ is not Gorenstein and hence $A$ is not Gorenstein. Let $\omega_A$ be the canonical module of $A.$ Note that the ideal $J=(x)$ is a minimal reduction of $\m$ with $r(\m)=4.$ By Proposition \ref{rossi} and \ref{bound}, we have $r(\m^i,\omega_A)=(\m^{i+3}\omega_A:\m^3\omega_A)$ for all $i=1,2,3$ and $r(\m^i,\omega_A)=(\m^{i}\omega_A:\omega_A)$ for all $i\geq 4.$
    
    Using Macaulay $2$, we can check that $G_{\m}(A)$ is not unmixed.
    Using Singular, one can check that the canonical module of $A$ is isomorphic to a submodule of $A^2$ and is generated by $\big((x,-y^2),(y,-z),(z,-x^4)\big)$ with $$(\m^5\omega_A:\m^3\omega_A)=(z,x^2,xy,y^2)\neq \m^2$$ and  $$(\m^6\omega_A:\m^3\omega_A)=(yz,x^3,x^2y,xy^2,y^3)\neq \m^3.$$ We also observed that $r(\m^4,\omega_A)=(x^4,x^{3}y,x^{2}y^2,xy^3,y^4)\subseteq\m^4$. Since $\m^4\subseteq r(\m^4,\omega_A)$ we have $r(\m^4,\omega_A)=\m^4$. 

    \textbf{Claim:} $r(\m^i,\omega_A)=\m^i$ for all $i\geq 4.$

    Fix $i>4$. Let $\alpha \in r(\m^i,\omega_A)=(\m^{i}\omega_A:\omega_A)$. This implies that $\alpha \omega_A\subseteq \m^i\omega_A$. Since $r(\m)=4$ we have $$\alpha \omega_A\subseteq x^{i-4}\m^4\omega_A.$$ By going modulo $(x^{i-4})$, we get $\bar{\alpha}\bar{\omega_A}=0.$ Since $\bar{\omega_A}=w_{A/x^{i-4}}$ (see \cite{Bruns}, Theorem 3.3.5) and $w_{A/x^{i-4}}$ is a faithful $A/(x^{i-4})$-module (see \cite{Bruns}, Theorem 3.3.13), we obtain $\alpha \in (x^{i-4}).$ Let $\alpha=x^{i-4}t$. Since $x^{i-4}$ is $\omega_A$-regular and  from the inclusion $\alpha \omega_A\subseteq x^{i-4}\m^4\omega_A$, we obtain $t\in (x^{i-4}\m^4\omega_A:x^{i-4}\omega_A)=(\m^4\omega_A:\omega_A)=\m^4$. Hence $\alpha \in \m^i.$ This proves the claim.
    
    This implies that neither $E=0$ nor $A$ is a quasi-Hilbert ring (see Theorem \ref{Thm:main}).  
\end{example}

\section{Applications of quasi-Hilbert ring}

In this section, we discuss some applications of a quasi-Hilbert ring.
\begin{theorem}\label{hil}
    Let $(A,\m)$ be a Cohen-Macaulay local ring of dimension $d\geq 1$. Suppose $\omega_A$ is the canonical module of $A.$ Let $L$ be an $A$-module such that $\operatorname{ann}(L)=0$ and for some $r\geq 1$, we have a surjective map  $$\omega^r_A\to L.$$ If $A$ is a quasi-Hilbert ring then the function $n\mapsto \ell(r(\m^n,L)/\wt{\m^n})$ coincides with a polynomial of degree $d-1.$
\end{theorem}
\begin{proof}
    By Proposition $1.4(a,b)$ of \cite{Zulfeqarr}, we obtain $r(\m^n,w^r_A)=r(\m^n,\omega_A)\subseteq r(\m^n,L)$ for all $n\geq 1.$ Since $\m^n\subseteq r(\m^n,\omega_A)$, we have the following injective map $$r(\m^n,\omega_A)/\wt{\m^n}\to r(\m^n,L)/\wt{\m^n}$$ for all $n\geq 1.$ It follows that the growth of the function $n\mapsto \ell(r(\m^n,L)/\wt{\m^n})$ is at least as that of the the function $n\mapsto \ell(r(\m^n,\omega_A)/\wt{\m^n})$.  Since $A$ is a quasi-Hilbert ring, the function $n\mapsto \ell(r(\m^n,\omega_A)/\wt{\m^n})$ coincides with a polynomial of degree $d-1.$ This implies that the function $n\mapsto \ell(r(\m^n,L)/\wt{\m^n})$ coincides with a polynomial of degree $d-1.$
\end{proof}
\begin{corollary}
     Let $(A,\m)$ be a Cohen-Macaulay local ring of dimension $d\geq 1$. Suppose $\omega_A$ is the canonical module of $A.$ Let $L$ be an injective $A$-module and $\operatorname{ann}(L)=0.$  If $A$ is a quasi-Hilbert ring then the function $n\mapsto \ell(r(\m^n,L)/\wt{\m^n})$ coincides with a polynomial of degree $d-1.$
\end{corollary}
\begin{proof}
    Since $L$ is injective $A$-module, by (\cite{Bruns}, Exercise 3.3.28(b)), we have a surjective map $$\omega^r_A\to L$$ for some $r>0.$ The result now follows from Theorem \ref{hil}.
\end{proof}

\begin{corollary}\label{max} 
     Let $(A,\m)$ be a Cohen-Macaulay local domain of dimension $d\geq 1$. Suppose $\omega_A$ is the canonical module of $A.$ Let $M$ be a maximal Cohen-Macaulay $A$-module and $N_i=\operatorname{Syz}^A_i(M)$ be the $i$th syzygy of $M$.  If $A$ is a quasi-Hilbert ring then the function  $n\mapsto \ell(r(\m^n,N^{\dagger}_i)/\wt{\m^n})$ coincides with a polynomial of degree $d-1$ for all $i\geq 1$, where $N^{\dagger}_i=\Hom(N_i,\omega_A).$
\end{corollary}
\begin{proof}
   Note that the annihilator of a maximal Cohen-Macaulay module over a Cohen-Macaulay domain is zero because the set of associated primes of $M$ is a subset of the set of minimal primes of $A.$ Set $N_0=M.$ For each $i\geq 0$, we have the following short exact sequence $$0\to N_{i+1}\to A^{\mu(N_i)}\to N_i\to 0.$$ After applying the functor $\Hom(-,\omega_A)$ and by (\cite{Bruns}, Theorem 3.3.10), we get $N^{\dagger}_{i+1}$ is a maximal Cohen-Macaulay $A$-module and $$0\to N^\dagger_i\to \omega_A^{\mu(N_i)}\to N^{\dagger}_{i+1}\to 0.$$ The result now follows from Theorem \ref{hil}.
\end{proof}

\begin{lemma}\label{contain}
     Let $(A,\m)$ be a local ring with an infinite residue field and $I$ a regular ideal of $A$. Let $x\in I\setminus I^2$ be a $M$-superficial element. Then 
     \begin{enumerate}
         \item For all $n\geq 1$, $$(r(I^n,M)+(x))/(x)\subseteq r(I^n,M/xM)=r((I^n+(x))/(x),M/xM).$$
         \item For all $n\geq 1$, $$\dfrac{(r(I^n,M)+(x))/(x)}{(\wt{I^n}+(x))/(x)}\cong \dfrac{r(I^n,M)}{\wt{I^n}+xr(I^{n-1},M)}.$$
     \end{enumerate}
\end{lemma}
\begin{proof}
(1) Choose $t\gg0$ such that $r(I^n,M)=(I^{n+t}M:I^tM)$ and $r(I^n,M/xM)=(I^{n+t}(M/xM):I^t(M/xM)).$ Let $a+(x)\in (r(I^n,M)+(x))/(x),$ where $a\in r(I^n,M).$ 
\begin{equation*}
    \begin{split}
        (a+(x))I^t(M/xM)&= aI^t(M/xM)+(x)I^t(M/xM)\\&= aI^t(M/xM)\\&=(aI^tM+xM)/xM\\&\subseteq (I^{n+t}M+xM)/xM\\&=I^{n+t}(M/xM).
    \end{split}
\end{equation*}
Hence $a+(x)\in r(I^n,M/xM).$

 Next, choose $t\gg0$ such that $$ r((I^n+(x))/(x),M/xM)=((I^{n+t}+(x)/(x))M/xM:(I^{t}+(x)/(x))M/xM)$$ and $$r(I^n,M/xM)=(I^{n+t}(M/xM):I^t(M/xM)).$$ It follows that
\begin{equation*}
\begin{split}
    r((I^n+(x))/(x),M/xM)&=((I^{n+t}+(x)/(x))M/xM:(I^{t}+(x)/(x))M/xM)\\&=((I^{n+t}M+xM)/xM:(I^{t}M+xM)/xM)\\&=(I^{n+t}(M/xM):I^t(M/xM))\\&=r(I^n,M/xM).
    \end{split}
\end{equation*}

(2) \begin{equation*}
         \begin{split}
             \dfrac{(r(I^n,M)+(x))/(x)}{(\wt{I^n}+(x))/(x)}& \cong (r(I^n,M)+(x))/(\wt{I^n}+(x))\\&\cong r(I^n,M)/(\wt{I^n}+(x))\cap r(I^n,M)\\&=r(I^n,M)/\wt{I^n}+((x)\cap r(I^{n},M))\\&=r(I^n,M)/\wt{I^n}+xr(I^{n-1},M).
         \end{split}
     \end{equation*}
     The last equality is because of (\cite{Zulfeqarr}, Proposition 8.2). 
\end{proof}
\begin{theorem} \label{modx}
     Let $(A,\m)$ be a Cohen-Macaulay local ring of dimension $d\geq 2$ with infinite residue field. Suppose $\omega_A$ is a canonical $A$-module. Let $x\in \m\setminus\m^2$ be an $A\bigoplus \omega_A$-superficial element. If $A$ is quasi-Hilbert, then so is $A/(x)$.    
\end{theorem}
\begin{proof}
    Note that the Jacobson ideal $J=\m$ and $$E=\bigoplus_{n\geq 0}\dfrac{r(\m^n,\omega_A)}{\wt{\m^n}}$$ is a finitely generated $\R^{\m}$-module. By (\cite{Zulfeqarr}, Proposition 8.2), for each $n\geq 1$ we have a  short exact sequence $$0\to \dfrac{r(\m^n,\omega_A)}{\wt{\m^n}}\xrightarrow{x}\dfrac{r(\m^{n+1},\omega_A)}{\wt{\m^{n+1}}}\to \dfrac{r(\m^{n+1},\omega_A)}{xr(\m^{n},\omega_A)+\wt{\m^{n+1}}}\to 0.$$ Passing to the direct sum over all $n\geq 0$ we obtain a short exact sequence of $\R^{\m}$-modules $$0\to E\xrightarrow{xt}E(+1)\to C=\bigoplus_{n\geq 0}\dfrac{r(\m^{n+1},\omega_A)}{xr(\m^{n},\omega_A)+\wt{\m^{n+1}}}\to 0.$$ Since $A$ is a quasi-Hilbert ring, the dimension of $E$ is $d$. Furthermore, $xt$ is a non-zero divisor on $E$ and hence the dimension of $C$ is $d-1.$ By Lemma \ref{contain}, for all $n\geq 1$,
    \begin{equation}\label{eqq}
        \begin{split}
            \dfrac{r(\m^n,\omega_A)}{\wt{\m^n}+xr(\m^{n-1},\omega_A)}&=\dfrac{(r(\m^n,\omega_A)+(x))/(x)}{(\wt{\m^n}+(x))/(x)}\\&\subseteq \dfrac{r((\m^n+(x))/(x),w_{A/x})}{(\wt{\m^n}+(x))/(x)}. 
        \end{split}
    \end{equation}
 Set $$D=\bigoplus_{n\geq 0}\dfrac{r((\m^n+(x))/(x),w_{A/x})}{\wt{\m^n}+(x)/(x)}.$$  Since $w_{A/x}$ is a faithful $A/x$-module and $\m/x$ is a regular ideal, the module $D$ is a finitely generated $\R^{\bar{\m}}$-module, where $\bar{\m}=(\m+(x))/(x).$ From $\ref{eqq}$, we can conclude that the dimension of the $\R^{\bar{\m}}$-module $D$ is $d-1.$

    Next, observe that for all $n\geq 1$, $$\dfrac{\wt{\m^n}+(x)}{(x)}\subseteq \wt{\bar{\m^n}}.$$ This inclusion yields a short exact sequence of $\R^{\bar{\m}}$-modules. $$0\to \bigoplus_{n\geq 0}\dfrac{\wt{\bar{\m^n}}}{\wt{\m^n}+(x)/(x)}\to \bigoplus_{n\geq 0 }\dfrac{r({\bar{\m}^n},w_{A/x})}{\wt{\m^n}+(x)/(x)}\to \bigoplus_{n\geq 0 }\dfrac{r(\bar{\m^n},w_{A/x})}{\wt{\bar{\m^n}}}\to 0.$$ Since $d\geq 2$ and $\wt{\bar{\m^n}}=\wt{\m^n}+(x)/(x)$ for all $n\gg0$, we get the dimension of $\R^{\bar{\m}}$-module $\bigoplus_{n\geq 0 }\dfrac{r(\bar{\m^n},w_{A/x})}{\wt{\bar{\m^n}}}$ is $d-1$. Therefore, $A/(x)$ is a quasi-Hilbert ring. 
\end{proof}

%\textcolor{blue}{The hypothesis of Theorem \ref{modx} holds if $A$ is a Cohen-Macaulay local UFD with an infinite residue field. In fact, in this ring, every linear superficial element is a prime element. Hence, $A/x$ is a domain.}

\begin{theorem}\label{hat}
      Let $(A,\m)$ be a Cohen-Macaulay local ring of dimension $d\geq 1$. Suppose $\omega_A$ is a canonical $A$-module. $A$ is quasi-Hilbert iff $\hat{A}$ is quasi-Hilbert, where $\hat{A}$ is the completion of $A$ with respect to $\m$.
\end{theorem}
\begin{proof}
Note that the Jacobson ideal $J=\m$. Since the map $A \to \widehat{A}$ is a flat local homomorphism and $\mathfrak m \widehat{A}$ is the maximal ideal of $\widehat{A}$, it follows from Proposition~3.3.14 of \cite{Bruns} that

$$\omega_A \otimes_A \widehat{A}$$

is a canonical $\widehat{A}$-module.

Moreover, by \cite{Zulfeqarr}, Proposition~1.4(d), we have an isomorphism

$$r(\mathfrak m^{n}, \omega_A)\otimes_A \widehat{A}
\;\cong\;
r(\mathfrak m^{n}\widehat{A},\, \omega_A\otimes_A \widehat{A}).$$

By \cite{Bruns}, Corollary~2.1.8, $A$ is a Cohen-Macaulay local ring of dimension $d\geq 1$ iff $\widehat{A}$ is a Cohen-Macaulay local ring of dimension $d \geq 1$. Hence, the filtration $\F^{\m}_{\omega_A}$ is $\m$-stable iff the filtration $$\F^{\m\hat{A}}_{\omega_A\otimes \hat{A}}=\{\, r(\mathfrak m^{n}\widehat{A},\, \omega_A\otimes_A \widehat{A}) \,\}_{n\ge 0}$$
forms an $\mathfrak m\widehat{A}$-stable filtration.

We also observe that
\[
\frac{\, r(\mathfrak m^{n}\widehat{A},\, \omega_A\otimes_A \widehat{A}) \,}
     {\, \widetilde{\mathfrak m^{n}\widehat{A}}\,}
\;\cong\;
\frac{\, r(\mathfrak m^{n}, \omega_A) \,}{\, \widetilde{\mathfrak m^{n}}\,}
     \otimes_A \widehat{A}.
\]
Since the length of an $A$-module is preserved under faithfully flat extension (see \cite{Part1}, 1.4), it follows that $A$ is quasi-Hilbert iff $\hat{A}$ is quasi-Hilbert

\end{proof}
%\pagebreak

\section{Vanishing of $r(I^n,M)/\wt{I^n}$ for all $n\geq 1$}

In this section, we present sufficient conditions under which the equality $r(I^n,M)=\wt{I^n}$ holds for all 
$n\geq 1$ where $I$ is a regular ideal and $M$ is a maximal Cohen–Macaulay $A$-module.

\begin{theorem}\label{vanish}
    Let $(A,\m)$ be a Cohen-Macaulay local ring of dimension $d\geq 1$ and $M$ be a maximal Cohen-Macaulay $A$-module with $\operatorname{ann}(M)=0$. Let $I$ be a regular radical ideal of $A.$ If $\wt{G}_I(A)$ is unmixed then the following holds.
    \begin{enumerate}
        \item $E=\bigoplus_{n\geq 1}r(I^n,M)/\wt{I^n}$ has dimension equal to $d$ or $E=0$ as $\R^I$ as well as $\wt\R^I$-module.
        \item Suppose $E\neq 0.$ Let $\q$ be a prime ideal of $A$ containing $I.$ Then the $\dim(E_{\q})=\operatorname{ht}(\q)$ as a module over $\widehat{\wt{\R}^{I_{\q}}}=\bigoplus_{n\in \Z}\wt{I^n_{\q}}t^n,$ where $\wt{I^n_{\q}}=A_{\q}$ for all $n\leq 0.$ 
        \item If $M_{\q}$ is free for some prime $\q$ containing $I$ then $r(I^n,M)=\wt{I^n}$ for all $n\geq 1.$ 
    \end{enumerate}
\end{theorem}
\begin{proof}
   (1) Since $I$ is a regular ideal and ann$(M)=0$ the filtrations $\F^I_M$ and $\{\wt{I^n}\}_{n\geq 0}$ are $I$-stable. It follows that $E$ is a finitely generated $\R^I$-module. Note that $\wt{G}_I(A)$ is unmixed and equidimensional (see \ref{UE}). Then by the proof of Theorem 3.5 of \cite{samunmixed}, $\dim(E)=d$ as an $\R^I$-module. Also, note that $\R(\F^I_M)$ is a graded $\wt\R^I=\bigoplus_{n\geq 0}\wt{I^n}t^n$-module. From the containment $$\R^I\subseteq \wt\R^I\subseteq\R(\F^I_M),$$ it follows that $E$ is a finitely generated $\wt\R^I$-module. Since $\wt{G}_I(A)$ is unmixed and equidimensional, by the proof of Lemma 3.3 and Theorem 3.5 of \cite{samunmixed}, we obtain that the dimension of $E$ is $d$ as an $\wt\R^I$-module.

   (2) Note that $I^n\subseteq \wt{I^n}\subseteq r(I^n,M)\subseteq \overline{I}=I,$ where $\overline{I}$ is the integral closure of $I.$ The last equality due to $I$ is a radical ideal, and the second last inequality is due to Proposition 4.2 of \cite{Zulfeqarr}. Therefore, they all have the same minimal primes. Suppose $E\neq 0.$ This implies that $E_{\q}\neq 0$ for all prime ideal $\q$ containing $I.$
   
   Let $\q$ be a  prime ideal containing $I$. This implies $I_{\q}$ is a regular ideal of $A_{\q}$. We also have $\text{ann}_{A_{\q}}(M_{\q})=(\text{ann}_A(M))_{\q}=0.$ Hence the filtration $\F^{I_{\q}}_{M_{\q}}$ is $I_{\q}$-stable. It follows that $E_{\q}$ is a finitely generated $\wt{\R}(I_{\q})$-module. 
   
   %\textbf{Claim} $E$ is a finitely generated graded $\hat{\wt\R}=\bigoplus_{n\in \Z}\wt{I^n}t^n$-module, where $\wt{I^n}=A$ for all $n\leq 0$.

Note that
$$\widetilde{I^n}\, r(I^m,M)\subseteq r(I^{n+m},M)
\quad \text{for all } n\in \mathbb{Z} \text{ and } m\ge 0.$$
It follows that $E$ is a graded $\widehat{\widetilde{\mathcal R}^I}=\bigoplus_{n\in \Z}\wt{I^n}t^n$-module.

Since $E$ is a finitely generated $\widetilde{\mathcal R}^I$-module, it follows that
for all $n\gg 0$, the ideal $(t^{-n})$ lies in the annihilator of $E$ as a
$\widehat{\widetilde{\mathcal R}^I}$-module. Consequently, $E$ is a module over
$\widehat{\widetilde{\mathcal R}^I}/(t^{-n})$ for all $n\gg 0$.

Let $\mathfrak p_1,\ldots,\mathfrak p_s$ be the minimal primes of $(t^{-1})$.
For $n\gg 0$, let
$$(t^{-n})
=
Q_{1,n}\cap\cdots\cap Q_{s,n}\cap L_{1,n}\cap\cdots\cap L_{t,n}$$
be a primary decomposition of $(t^{-n})$, where each $Q_{i,n}$ is $\mathfrak p_i$-primary.

\medskip
\noindent
\textbf{Claim.}
For all $n\gg 0$, we have $Q_{i,n}\neq \mathfrak p_i$ for every $i=1,\ldots,s$.

\medskip
\noindent
\textit{Proof of the claim.}
Suppose that $Q_{i,n}=\mathfrak p_i$ for some $i$ and for infinitely many
values of $n$. Let $\Delta$ be the index set such that
$$Q_{i,n}=\mathfrak p_i \quad \text{for all } n\in \Delta.$$
Then
$$\mathfrak p_i
=
Q_{i,n}
=
(t^{-n})(\widehat{\widetilde{\mathcal R}})_{\mathfrak p_i}\cap
\widehat{\widetilde{\mathcal R}^I}
\quad \text{for all } n\in \Delta.$$
Localizing at $\mathfrak p_i$, we obtain
$$(\mathfrak p_i)_{\mathfrak p_i}
=
(t^{-n})(\widehat{\widetilde{\mathcal R}^I})_{\mathfrak p_i}
\quad \text{for all } n\in \Delta.$$
Hence,
$$(\mathfrak p_i)_{\mathfrak p_i}
=
\bigcap_{n\in \Delta}
(t^{-n})(\widehat{\widetilde{\mathcal R}^I})_{\mathfrak p_i}
=
0.$$
This implies that $\mathfrak p_i=0$, which contradicts the fact that $\mathfrak p_i\neq 0$. Therefore, the claim follows.

{Since $\wt{G}_I(A)\cong \widehat{\wt{\R}^I}/(t^{-1})$ and from the surjective map $$\widehat{\wt{\R}^I}/(t^{-n})\to \widehat{\wt{\R}^I}/(t^{-1}),$$ we obtain the dimension of $E$ is $d=\dim(\widehat{\wt\R^I}/(t^{-n}))$ as a $\widehat{\wt\R^I}$-module. Hence there exists $\p_1$(say) such that $\p_1\in \text{Supp}_{\widehat{\wt\R^I}}(E).$ Note that height of ${\p_1}$ is one.}

\medskip
\noindent
\textbf{Claim.} ${{\mathfrak p_1}}_{\mathfrak q}\neq 0$.

%\medskip
{If the claim holds, then by a graded analogue of \cite[Lemma 2, p.\ 250]{matsumura}, we have $\text{height}({\mathfrak{M_{\q}}}/{{\p_1}}_{\q}) = \text{height}\ {\mathfrak{M_{\q}}} - \text{height }  {\p_1}_{\q} = \text{ht}(\q)+1 -1 = \text{ht}(\q)$, where $\mathfrak{M_{\q}}$ is the graded maximal ideal of $\hat{\wt\R}(I_{\q})$. This implies that the $\dim(E_{\q})=\text{ht}(\q)$ as a module over $\hat{\wt{\R}}(I_{\q}).$}

\medskip
\noindent
\textit{Proof of the claim.}
Suppose, to the contrary, that $\bar{{\mathfrak p_1}}_{\mathfrak p}=0$.
Then there exists $s\notin \mathfrak p$ such that
$s\bar{\mathfrak p_1}=\bar{0}=(t^{-n})=Q_{1,n}\cap\cdots\cap Q_{s,n}\cap L_{1,n}\cap\cdots\cap L_{t,n}$.
Note that for $i\ge 2$, we have
$\mathfrak p_i=\sqrt{Q_{i,n}}\not\supseteq \mathfrak p_1$.
Since $Q_{i,n}$ is $\mathfrak p_i$-primary, it follows that
$s\in \mathfrak p_i$ for all $i\ge 2$.

Choose $\alpha\in \mathfrak p_1\setminus Q_{1,n}$ such that
$s\alpha\in Q_{1,n}$.
Since $Q_{1,n}$ is $\mathfrak p_1$-primary, we conclude that
$s\in \mathfrak p_1$.
Hence,
$s\in \mathfrak p_1\cap\cdots\cap \mathfrak p_s=\sqrt{(t^{-1})}$.
Therefore, there exists $m\in \mathbb{Z}$ such that
$s^m\in (t^{-1})$.

Now we have a surjective homomorphism
$$\widehat{\widetilde{\mathcal R}^I}/(s^m)\to
\widehat{\widetilde{\mathcal R}^I}/(t^{-1}).$$
Since degree of $s$ is zero, by focusing on $n$-th component of the homomorphism, we obtain
$s^m\widetilde{I^{n}}\subseteq \widetilde{I^{n+1}}$.
Consequently,
$s^m\in (\widetilde{I^{n+1}}:\widetilde{I^{n}})\subseteq \widetilde{I}=I$,
where the last equality holds since $I$ is a radical ideal.
Thus,
$s^m\in I\subseteq \mathfrak p$, and hence $s\in \mathfrak p$,
which is a contradiction.
This proves the claim.

(3) Suppose $E\neq 0.$ By part $(2)$, dimension of $E_{q}=\text{ht}(q)\geq 1.$ By the hypothesis, $E_{\q}=\bigoplus_{n\geq 0}r(I^n_{\q},M_{\q})/\wt{I^n}_{\q}=0$. Which is a contradiction to part $(2).$ This proves the result.

\end{proof}

\begin{theorem}\label{v2}
    (with the same hypothesis of Theorem \ref{vanish}) If $G_I(A)$ is unmixed then the following holds.
    \begin{enumerate}
        \item $E=\bigoplus_{n\geq 1}r(I^n,M)/\wt{I^n}$ has dimension equal to $d$ or $E=0$ as $\R^I$-module.
        \item Suppose $E\neq 0.$ Let $\q$ be a prime ideal of $A$ containing $I.$ Then the $\dim(E_{\q})=\operatorname{ht}(\q)$ as a module over $\widehat{{\R}^{I_{\q}}}=\bigoplus_{n\in \Z}{I^n_{\q}}t^n,$ where $\wt{I^n_{\q}}=A_{\q}$ for all $n\leq 0.$ 
        \item If $M_{\q}$ is free for some prime $\q$ containing $I$ then $r(I^n,M)=\wt{I^n}$ for all $n\geq 1.$ 
    \end{enumerate}
\end{theorem}
\begin{proof}
   The argument is similar to the proof of Theorem~\ref{vanish}, and hence we omit the details.
\end{proof}

\begin{corollary}
    Let $A$ be a Cohen-Macaulay ring of dimension $d\geq 1$ having canonical module $\omega_A.$ Let $\operatorname{ht}(\m)=d=\dim(A)$ for all maximal ideal $\m.$ Let $I=\q_1\cap \ldots \cap q_s$ be a regular radical ideal of $A$ such that $1\leq \operatorname{ht}(\q_i)<d.$ Then $r(I^n,\omega_A)=\wt{I^n}$ for all $n\geq 1$ if the following conditions hold:
    \begin{enumerate}
        \item $A_{q_i}$ is Gorenstein for all $i=1,\ldots ,s.$
        \item $G_{I}(A)$ is unmixed.
    \end{enumerate}
\end{corollary}

\begin{proof}

Let
$
E=\bigoplus_{n\geq 1} {r(I^n,\omega_A)}/{\widetilde{I^n}}.
$
Suppose that $E\neq 0$. Then there exists a maximal ideal $\mathfrak m$ such that
$
E_{\mathfrak m}
=\bigoplus_{n\geq 1}
{r(I^n_{\mathfrak m},(\omega_A)_{\mathfrak m})}/
{\widetilde{I^n_{\mathfrak m}}}
\neq 0.
$
Since $G_I(A)$ is unmixed and equidimensional (see Proposition~\ref{UE}(1)), {it follows that
$G_{I_{\mathfrak m}}(A_{\mathfrak m})$ is also unmixed and equidimensional for every
maximal ideal $\mathfrak m$.} Hence, from $E_{\mathfrak m}\neq 0$, we obtain
$E_{\mathfrak q_i}\neq 0$ for some $1\le i\le s$.

By part~(2) of Theorem~\ref{v2}, we have
$
\dim(E_{\mathfrak q_i})=\operatorname{ht}(\mathfrak q_i)\ge 1.
$
Thus $\dim(E_{\mathfrak q_i})>0$ as a module over the extended Rees algebra
$
\widehat{\mathcal R(I_{\q_i})}
=\bigoplus_{n\in\mathbb Z} I^n_{\q_i} t^n,
$
where $I^n_{\q_i}=A_{\q_i}$ for all $n\le 0$.

On the other hand, since $A_{\mathfrak q_i}$ is Gorenstein, we have
$$
E_{\mathfrak q_i}
=\bigoplus_{n\geq 1}
{r(I^n_{\mathfrak q_i},(\omega_A)_{\mathfrak q_i})}/
{\widetilde{I^n_{\mathfrak q_i}}}
=
\bigoplus_{n\geq 1}
{r(I^n_{\mathfrak q_i},A_{\mathfrak q_i})}/
{\widetilde{I^n_{\mathfrak q_i}}}
=0,
$$
which is a contradiction. Therefore, $E=0$.

\end{proof}

\providecommand{\bysame}{\leavevmode\hbox to3em{\hrulefill}\thinspace}
\providecommand{\MR}{\relax\ifhmode\unskip\space\fi MR }
 %\MRhref is called by the amsart/book/proc definition of \MR.
\providecommand{\MRhref}[2]{
  \href{http://www.ams.org/mathscinet-getitem?mr=#1}{#2}
}

\end{document}